\documentclass[reqno]{amsart}

\usepackage[utf8]{inputenc}
\usepackage[T1]{fontenc}

\usepackage{amsmath}
\usepackage{amssymb}
\usepackage{amsfonts}
\usepackage{graphicx}
\usepackage{amsthm}
\usepackage{enumerate}
\usepackage{lscape}
\usepackage{dsfont}
\usepackage{color}
\usepackage{mathtools}

\usepackage{setspace}
\newcommand{\Z}{\mathds{Z}}
\newcommand{\Q}{\mathds{Q}}              
\newcommand{\R}{\mathds{R}}

\newcommand{\CP}{\mathds{C}\mathrm{P}}
\newcommand{\N}{\mathds{N}}

\newcommand{\CH}{\mathds{C}\mathrm{H}}

\newcommand{\C}{\mathds{C}}            
\newcommand{\de}{\partial}          

\newcommand{\K}{K\"{a}hler}

\def\b{\beta}

\def\b1{{\rm id}}

\newtheorem{theor}{Theorem}[section]
\newtheorem{prop}[theor]{Proposition}
\newtheorem{defin}{Definition}
\newtheorem{lem}[theor]{Lemma}
\newtheorem{cor}{Corollary}

\newtheorem{example}{Example}
\newtheorem{rmk}{Remark}

\begin{document}

\title[KRS  induced by infinite dimensional complex space forms]{K\"{a}hler-Ricci solitons induced by infinite dimensional complex space forms}

\author{Andrea Loi}
\address{(Andrea Loi) Dipartimento di Matematica \\
         Universit\`a di Cagliari, Via Ospedale 72, 09124  (Italy)}
         \email{loi@unica.it}

\author{Filippo Salis}
\address{(Filippo Salis) Dipartimento di Scienze Matematiche “G. L. Lagrange”, Politecnico di Torino,
Corso Duca degli Abruzzi 24, 10129 Torino (Italy)}
\email{filippo.salis@polito.it}

\author{Fabio Zuddas}
\address{(Fabio Zuddas) Dipartimento di Matematica \\
         Universit\`a di Cagliari, Via Ospedale 72, 09124  (Italy)}
         \email{ fabio.zuddas@unica.it}

\thanks{
The first and the third authors were supported  by STAGE - Funded by Fondazione di Sardegna and by KASBA- Funded by Regione Autonoma della Sardegna. The second author was supported by PRIN 2017 “Real and Complex Manifolds: topology, geometry and holomorphic dynamics” and MIUR grant “Dipartimenti di 
Eccellenza 2018-2022”. All the three authors were supported by INdAM GNSAGA - Gruppo Nazionale per le Strutture Algebriche, Geometriche e le loro Applicazioni.
}


\begin{abstract}
We exhibit  families of  non trivial (i.e. not \K-Einstein) radial  \K-Ricci solitons (KRS), both complete and not complete,  which can be \K\ immersed into   infinite dimensional complex space forms.  This result shows that  the triviality of a KRS induced by a finite dimensional complex space form proved in \cite{LM} does not hold when the ambient space is allowed to be infinite dimensional. Moreover, we show that the radial potential of a   radial  KRS induced by a non-elliptic complex space form is necessarily  defined at the origin.
\end{abstract}
 
\maketitle

\tableofcontents  

\section{Introduction}
The study of those complex manifolds $M$ equipped with a \K-Einstein (KE) metric $g$ induced by a complex space form, namely such that $(M, g)$  can be \K\ immersed\footnote{Throughout the paper the \K\ manifold $M$ is not necessarly  compact (or complete) and the \K\ immersion is not required to be injective  or an embedding.} into
a finite or infinite dimensional complex space form $(S, g)$, is a classical problem in complex differential geometry. 
When the ambient complex space form is assumed to be finite dimensional and of non positive holomorphic sectional curvature M. Umehara \cite{UmearaE} shows that  $(M, g)$ is forced to be totally geodesic and hence is  itself an open subset of  a complex space form. On the other hand a classification of those KE manifolds \K\ immersed into  the  finite dimensional complex projective space is still missing. The general conjecture is that such a KE manifold is  forced to be an open subset of a compact  homogeneous \K\  manifold, i.e. it is acted upon transitively by its group of holomorphic isometries. 
Roughly speaking when the ambient space is finite dimensional one has (locally) a finite number of holomorphic functions describing the \K\ immersion which seems to force the potential of a KE metric  
(which satisfies a  Monge-Ampere equation) to have symmetries, i.e. to be the potential of a homogeneous metric.
Many authors have proved the validity of this conjecture under additional assumptions  
(see, e.g.  \cite{Ch}, \cite{Sm}, \cite{ts}, \cite{Ha}, \cite{hulinlambda}, \cite{MS}).
When the ambient space is infinite dimensional the situation  changes drastically:  there exist continuous families 
of complete not homogeneous  KE metrics  projectively induced by  
an infinite dimensional  complex projective space\footnote{We still do not know if  similar phenomena can also happen  in the infinite dimensional non elliptic case.} 
(see \cite{LZ} and \cite{HY}).

Therefore it is natural to  impose some extra conditions on the KE metric $g$  
in order to recover the loss of symmetries due to the  infinite dimensional assumption of the ambient space.
One natural  condition is  to require that the metric is radial,  i.e.  $g$ admits a global K\"ahler potential  $f(r)$ which depends only on the sum $r= |z|^2 = |z_1|^2 + \cdots + |z_n|^2$ of the local coordinates' moduli.
Notice that since the manifold $M$ is assumed to be connected the potential  $f(r)$ is defined on an open interval 
$(r_{\inf}, r_{\sup})$, $0<r_{\inf}<r_{\sup}$. 
The prototypes of radial KE  metrics  are of course the finite dimensional  complex space forms
 and  any homogeneous \K\ manifold with a radial potential is indeed a complex space form. 
The main  result in this regard  found by the authors of the present paper 
can be summarized as follows (see Definition \ref{WB} in the next section  and Definition \ref{defst} in Section \ref{projindKRS}   for the notions   of well-behaved or $c$-stable projectively induced metrics).

\vskip 0.3cm
\noindent
{\bf Theorem A.} (see \cite[Theorem 1.3 and Theorem 1.4]{LSZext})
{\em Let $g$ be a radial KE metric  on a complex manifold $M$ and assume that $(M, g)$ can be \K\ immersed into 
an infinite dimensional complex space form $(S^{\infty}, g^{\infty}_c)$  with constant holomorphic sectional curvature $c$.

(1) If  $c\leq 0$ then $(M, g)$
is a complex space form.

(2) If $c>0$  and the metric  $g$ is either  well-behaved or $c$-stable projectively induced 
then $(M, g)$ is a complex space form.}

\vskip 0.3cm

We believe that the assumptions  that $g$ is well-behaved or $c$-stable projectively induced in Theorem A are  superfluous (cfr. \cite[Conjecture 2]{LSZext}). 
This is true if the Einstein constant of $g$ vanishes: indeed in 
 \cite{LOIZEZU} we prove that   a projectively induced  Ricci flat metric is forced to be flat.
It is also  worth mentioning that the KE condition in Theorem A can be weakened  to  constant scalar curvature (cscK) case
\cite[Theorem 1.3]{LSZext} but not to the case of Calabi's extremal metrics, see \cite[Example 1]{LSZext}.
Both cscK and extremal metrics are generalization of KE metrics.
Another natural extension is that of  {\em Kahler-Ricci soliton} (KRS);
therefore  it is natural to study  radial KRS 
 induced by infinite dimensional complex space forms.
 This is what we do in the present paper.
 Recall that a KRS on a complex manifold $M$ is a pair $(g, X)$ consisting of a Kähler metric $g$ and a
holomorphic vector field $X$, called the {\em solitonic vector field}, such that 
\begin{equation}\label{eqkrsg}
\rho=\lambda \omega+L_{X}\omega
\end{equation}
for some $\lambda \in \mathbb{R}$, called the {\em solitonic constant}.
Here $\omega$ and $\rho$ are respectively the \K\ form and  the Ricci form of the metric $g$ 
and $L_X\omega$ denotes the Lie derivative of $\omega$ with respect to $X$.
KRS are  special solutions of the \K-Ricci flow 
and they generalize \K--Einstein (KE) metrics. Indeed any
KE metric $g$ on a complex manifold $M$ gives rise to a
trivial KRS by choosing $X = 0$ or $X$ Killing with
respect to $g$. Obviously if the automorphism group of $M$ is
discrete then a KRS $(g, X)$ is nothing but a
KE metric $g$. The reader is referred to \cite{KOISO, CAO, WAZHU, POSPIRO, tiansol1, tiansol2}
for more information on KRS.

It turns out that a  radial KRS with given solitonic constant $\lambda$ on a $n$-dimensional complex manifold  $M$
is  uniquely determined by  $(\mu, \nu, k)\in\R^3$  if $n=1$ and $(\mu, \nu)\in\R^2$ if $n\geq 2$ 
(cfr.  Proposition \ref{mainprop} in the next section).  
Further the KRS is not trivial if $\mu$ and $\nu$ are not zero.

The following theorem is the first main   result of the paper.
\begin{theor}\label{mainteor}
Let $(g, X)$ be a non trivial  KRS with solitonic constant $\lambda$ on an $n$-dimensional complex manifold $M$.
Assume that\footnote{It turns out that the  condition \eqref{condpar} 
is equivalent to the fact that  the \K\  potential  $f(r)$ of the metric $g$ is defined at the origin, namely at  $r_{\inf}=0$
(see Proposition \ref{mainprop} below).}
\begin{equation}\label{condpar}
\nu = \frac{n!(\mu-\lambda)}{\mu^{n+1}}.
\end{equation}
Then the following facts hold true.
\begin{itemize}
\item [(a)]
if $\mu$ and $\nu$ are strictly positive and $k=0$ then $(M, g)$ can be \K\ immersed into any  infinite dimensional
complex space forms of non-negative holomorphic sectional curvature.
\item [(b)]
if $\lambda\leq 0$,  $\mu=n+1$ and $k=0$ then $(M, g)$ can be \K\ immersed into any  infinite dimensional complex space form.  
\end{itemize}
\end{theor}

Theorem \ref{mainteor} shows that 
the  same conclusions of Theorem A (namely the constancy of the holomorphic sectional curvature) cannot be achieved  if one weakens  Einstein's condition  with that of  KRS, even if one requires that  the radial potential  $f(r)$ of the metric $g$ is defined at the origin  (which is stronger than well-behaveness) and that the metric $g$ is induced by any complex space form  (which is much stronger than $c$-stability, see Remark \ref{rmkstable} in Section \ref{projindKRS} below).
The theorem  also shows    that   the main result in \cite{LM} due to the first author and R. Mossa, 
asserting  that a KRS induced by a {\em finite} dimensional (even indefinite) complex space form is trivial,
does not extend to the infinite dimensional setting.

One can show that the KRS  in Theorem \ref{mainteor} are not complete, namely their \K\ metric are not complete.
Thus it is natural to see if there exist complete  KRS induced by  some infinite dimensional complex space form.
The following  theorem, our second main result,  shows that this is indeed  the case.  

\begin{theor}\label{mainteor3}
For any $n\geq 2$ there exist complete and radial 
KRS on $\C^n$ induced by any infinite dimensional complex space forms of non-negative holomorphic sectional curvature\footnote{
We do not know if there exist complete KRS induced by the  infinite dimensional complex hyperbolic space.}.
\end{theor}

We believe that the requirement \eqref{condpar} is a necessary condition for the \K\ metric $g$ to be induced by a complex space form, as expressed by  the following:

\vskip 0.3cm

\noindent
{\bf Conjecture 1:}
{\em The potential of a radial KRS induced by an infinite dimensional complex space form is defined at the origin.}

\vskip 0.3cm

In this regard we are able to prove the  following theorem which represents our third and last result.
\begin{theor}\label{mainteor2}
Let $(g, X)$ be a radial KRS on a complex manifold $M$  of complex dimension $n\geq 2$. 
If the metric $g$ is $c$-stable projectively induced, for some $c>0$, then its \K\ potential  $f(r)$ is defined at the origin and hence  \eqref{condpar} holds true.
\end{theor}

Since  a \K\ metric induced by a  non-elliptic complex space form is $c$-stable for any $c>0$ (see Remark \ref{rmkstable} in Section \ref{projindKRS} below), 
Theorem \ref{mainteor2}
gives the following corollary, which provides us with a partial answer to the conjecture and also  shows that
in order to prove its validity one can restrict to the case when the ambient complex space form
 is the complex projective space.

\begin{cor}
If a radial KRS on a complex manifold of complex dimension $n\geq 2$  is induced by either an infinite dimensional flat space or by an infinite dimensional complex hyperbolic space then its
\K\ potential is defined at the origin.
\end{cor}

The paper is organized as follows. In the next section we  describe the basic facts on radial \K\ metrics and we provide the classification  of radial  KRS (Proposition \ref{mainprop}).
In Section \ref{proofmainteor}, after recalling some necessary and sufficient  conditions  for a radial \K\ metric to be induced by a complex space form (Lemma \ref{calcrit})  we prove  Theorem \ref{mainteor}.
In Section \ref{compKRS} we   show  that the \K\ manifolds  appearing  in  the proof of Theorem \ref{mainteor} are not complete and  we prove Theorem \ref{mainteor3}.
Finally, Section \ref{projindKRS} is dedicated to the proof of  Theorem \ref{mainteor2}.
The paper  ends with an appendix  with two lemmata needed in the proof of Proposition \ref{mainprop}.

\section{Radial KRS}
Let $g$ be a radial  \K\ metric on a connected  complex manifold $M$, equipped with complex coordinates $z_1, \dots ,z_n$ and let $\omega$ and $\rho$ be respectively 
the \K\ form and the Ricci form associated to $g$.
Then  there exists a   smooth radial  function 
$$f: (r_{\inf}, r_{\sup})\rightarrow \R, \ 0\leq r_{\inf}<r_{\sup}\leq\infty,$$
where $(r_{\inf}, r_{\sup})$ is the maximal domain where $f(r)$ is defined 
such that
 \begin{equation}\label{omegar}
 \omega =\frac{i}{2} \partial \bar \partial f(r), \ r=|z|^2=|z_1|^2+\cdots +|z_n|^2,
 \end{equation}
 i.e. $f(r)$ is a radial potential for the metric $g$.

One can easily see that 
the matrix of the metric $g$ and of the Ricci form $\rho$ read  as
\begin{equation}\label{metric}
\omega_{i\bar j}=f'(r)\delta_{ij}+f''(r) \bar z_i z_j.
\end{equation}
\begin{equation}\label{formaricci}
\rho_{i\bar j}= L'(r) \delta_{i j} + L''(r) \bar z_ i z_j,
\end{equation}
where $L(r) = - \log(\det g)(r)$.

Set 
\begin{equation}\label{y(r)}
y(r):=rf'(r).
\end{equation}

\begin{defin}\label{WB}\rm
A radial \K\ metric $g$ is 
{\em well-behaved}  if $y(r)\rightarrow 0$ for $r\rightarrow r_{\inf}^+$. 
\end{defin}
Clearly  if a radial metric $g$ is defined at  $r_{\inf}=0$  then it is well-behaved.
In  particular any metric of constant holomorphic sectional curvature is well-behaved and 
even real analytic on $[0, r_{\sup})$.
Notice that it is not hard to see  that  a radial KE metric defined at the origin  is indeed  a complex space form.
Also set 
\begin{equation}\label{psi(y)}
\psi(r) := ry'(r).
\end{equation}
Then 
\begin{equation}\label{psi(y)2}
\psi(r) = \frac{dy}{dt},\  r=e^t.
\end{equation}

The fact that $g$ is a metric is equivalent to 
$y(r)>0$ and $\psi (r)>0$, $\forall r\in (r_{\inf}, r_{\sup})$.
Then 
\begin{equation}\label{limtr}
\lim_{r\rightarrow r^+_{\inf}}y(r)= y_{\inf}
\end{equation}
is a non negative real number.
Similarly set
\begin{equation}\label{limtrbis}
\lim_{r\rightarrow r^-_{\sup}}y(r)= y_{\sup}\in (0, +\infty].
\end{equation}

Therefore we can invert the map
$$(r_{\inf}, r_{\sup})\rightarrow  (y_{\inf}, y_{\sup}),\  r\mapsto y(r)=rf'(r)$$ 
on $(r_{\inf}, r_{\sup})$ and think $r$ as a function of $y$, i.e. $r=r(y)$.

Hence we can set
\begin{equation}\label{psi}
\psi (y):=\psi (r(y)).
\end{equation}

The following lemma will be used in the proof of Theorem \ref{mainteor2} below.
\begin{lem}\label{lemmalimpsi}
Assume that the function $\psi (y)$ is continuous at $y_{\inf}$.
If $\lim_{y\rightarrow y_{\inf}^+}\psi (y)\neq 0$
then $y_{\inf} =0$.
 \end{lem}
 \begin{proof}
Assume by contradiction that  $y_{\inf}\neq 0$.
Note first that $t_{\inf}:=\lim_{r\to r_{\inf}} \log r=-\infty$:
otherwise  (if  $t_{\inf}\in\R$) the function $y(t)$ could be prolonged to an open interval containing $t_{\inf}$
being 
 the solution of   the   Cauchy problem
\begin{equation}
\begin{cases}
 y'(t)=\psi (y(t)) \\
 y(t_{\inf})=y_{\inf}>0.
\end{cases}
\end{equation}
Thus, by  the continuity of $\psi (y)$ at $y_{\inf}\neq 0$,
$$\lim_{y\rightarrow y_{\inf}^+}\psi (y)=\lim_{t\rightarrow -\infty}\psi (y(t))=\lim_{t\rightarrow -\infty} y'(t)=0,$$
where the last equality follows by  \eqref{limtr} when $t_{\inf}=-\infty$,
the desired contradiction.
\end{proof}
 
Finally,  from (\ref{metric}), we easily get 
\begin{equation}\label{detmetric}
(\det g_{i\bar j})(r)=\frac{(y(r))^{n-1}\psi (y)}{r^{n}}.
\end{equation}

The following proposition, which represents a key tool in the proof of our main results,  
provide us with the  explicit expressions of  radial KRS  in terms of the the functions 
$y$ and $\psi (y)$ defined by \eqref{y(r)} and \eqref{psi(y)}.

 \begin{prop}\label{mainprop}
Let $g$ be a radial KRS with solitonic constant $\lambda$.
Then the following facts hold true.

If $n=1$ then there exist  $\mu, k\in\R $ such that

\begin{equation}\label{finalen=1}
\dot\psi(y)= \mu \psi(y) + k+1 -\lambda y
\end{equation}

and if $\mu=0$ then the  soliton is trivial (i.e. a complex space form).
If $\mu\neq 0$ then
\begin{equation}\label{psiespln=12}
\psi(y) = \nu e^{\mu y} + \frac{\lambda}{\mu} y + \left( \frac{\lambda}{\mu^2} - \frac{k + 1}{\mu} \right)
\end{equation}
and  the soliton is trivial iff it is flat iff $\nu=0$.

If $n\geq 2$ then there exists $\mu\in\R$ such that 

\begin{equation}\label{finale9}
\dot\psi(y) = \left( \mu - \frac{n-1}{y} \right) \psi (y) + n - \lambda y 
\end{equation}
and if $\mu=0$ the soliton is trivial  (i.e. KE).
If $\mu\neq 0$ then
\begin{equation}\label{psiesplicita}
 \psi(y) = \frac{\nu e^{\mu y}}{y^{n-1}} + \frac{\lambda}{\mu} y +  \frac{\lambda - \mu}{\mu^{1+n}} \sum_{j=0}^{n-1} \frac{n!}{j!} \mu^j y^{j+1-n}.
\end{equation}
and  the soliton is trivial iff it is flat iff $\nu=0$ and  $\mu =\lambda$.

Moreover, the  KRS is defined at the origin, i.e. at  $r_{\inf}=0$,
iff  $\nu = \frac{n!(\mu-\lambda)}{\mu^{n+1}}$ (namely \eqref{condpar} in Theorem \ref{mainteor} is satisfied).
\end{prop}

In order to prove the proposition we need the following two  technical lemmata whose proofs are relegated to Appendix \ref{proofslemma} below.

\begin{lem}\label{lemradial}
Let $G(z)= \Phi (z)+ \bar \Phi (z)$, where $\Phi(z)$, $z \in \C^n$, is a holomorphic function and $G (x_1, \dots ,x_n) = G(|z_1|^2, \dots, |z_n|^2)$ is a rotation invariant function, $x_j=|z_j|^2$. Then 
\begin{equation}\label{forlemmaradial}
G (x_1, \dots ,x_n) = \sum_{j=1}^n c_j \log x_j + d,
\end{equation} 
for some $c_j, d \in \R.$
In particular, if $G(x_1+\cdots +x_n) = G(|z_1|^2 + \cdots + |z_n|^2)$ is radial and $n \geq 2$, then the $c_j$'s must vanish and $G =d$ is constant.
\end{lem}

\begin{lem}\label{lemradial2}
Let the equality 
\begin{equation}\label{forlemma2}
\sum_{k=1}^n (\bar z_k Y_k(z) + z_k \bar Y_k(z)) = \phi(r)
\end{equation}
hold, where $Y_k$, $k = 1, \dots, n$, is a holomorphic function and $\phi(r)$ is a radial function. Then $\phi(r) = \alpha r$ for some $\alpha \in \R$.
\end{lem}
\begin{proof}[Proof of Proposition \ref{mainprop}]
Let $X = \sum_k X_k \frac{\partial}{\partial z_k} + \bar X_k \frac{\partial}{\partial \bar z_k}$ be a real holomorphic vector  field ($X_k=X_k(z)$ are holomorphic functions). 
If we take local complex coordinates $z_1, \dots , z_n$ and use the fact that 
$$L_X \omega (Y, Z) = X(\omega(Y,Z)) - \omega([X, Y], Z) - \omega(Y, [X, Z]).$$
we get
\begin{equation}\label{LXomega1}
(L_X \omega)_{i \bar j} =X(\omega_{i \bar j}) + \frac{\partial X_k}{\partial z_i} \omega_{k \bar j} +  \frac{\partial \bar X_k}{\partial \bar z_j} \omega_{i \bar k}
\end{equation}

By substituting \eqref{metric} in (\ref{LXomega1}) we obtain 
$$(L_X \omega)_{i \bar j}=(\sum_k X_k \bar z_k + \bar X_k z_k) (f''(r) \delta_{i j} + f'''(r) \bar z_ i z_j)  +f''(r)(X_j \bar z_i + \bar X_i z_j)+$$
$$\ \ \ \ \ \ \ \ \ \ \ \ \ \ \ +  \frac{\partial X_k}{\partial z_i} (f'(r)\delta_{k j} + f''(r) \bar z_ k z_j) +\frac{\partial \bar X_k}{\partial \bar z_j} (f'(r) \delta_{i k} + f''(r)\bar z_ i z_k).$$

Now, it is easy to see from a straight calculation that this last expression is equal to 
$$\frac{\partial^2}{\partial z_i \partial \bar z_j} \sum_k f' (r)(\bar z_k X_k + z_k \bar X_k)$$
and then the soliton equation \eqref{eqkrsg} writes
$$
\frac{\partial^2}{\partial z_i \partial \bar z_j} \sum_k f'(r) (\bar z_k X_k + z_k \bar X_k) = \rho_{i \bar j} - \lambda \omega_{i \bar j},
$$
By  taking into account \eqref{metric} and \eqref{formaricci}   one gets 
$$\frac{\partial^2}{\partial z_i \partial \bar z_j} \sum_k f'(r)(\bar z_k X_k + z_k \bar X_k) = a(r) \delta_{i j} + a'(r) \bar z_ i z_j,$$
where we set $a(r) := L'(r) -\lambda f'(r)$.

Let $\Psi:= \sum_k f' (\bar z_k X_k + z_k \bar X_k)$, we have
$$\frac{\partial^2 \Psi}{\partial z_i \partial \bar z_j}= a(r) \delta_{i j} + a'(r) \bar z_ i z_j$$
for every $i$ and $j$.
Thus
$$\frac{\partial^2 \Psi}{\partial z_i \partial \bar z_j}= \frac{\partial^2 A}{\partial z_i \partial \bar z_j}$$
where $A(r) = \int a(r) = L(r) - \lambda f(r) + \gamma$ ($\gamma \in \R$) and then one concludes that
\begin{equation}\label{FFbar}
\Psi = A(r) + F + \bar F
\end{equation}
for some holomorphic function $F$.
Thus we can write
$$-\log\det g-\lambda f+F+\bar F= \sum_kf'(\bar z_k X^k+z_k\bar X^k).$$
By averaging with respect to the action of  the unitary group $U(n)$, we get
\begin{equation}\label{shrexp}
-(\log\det g)(r)-\lambda f(r) + \int_{U(n)}\left(F (Az) +\bar F(Az) \right) dA =\sum_k f'(r) \left(\bar z_k Y^k(z)+z_k \bar Y^k(z)\right), 
\end{equation}
where $Y^k(z)=\int_{U(n)}\bar A_h^k X^h(Az) dA$. 

Equation (\ref{shrexp}) can be rewritten as  $\bar z_k Y^k+z_k \bar Y^k=\phi(r)$ where $\phi(r)$ is radial and the $Y^k$'s are holomorphic. Then Lemma \ref{lemradial2}  above applies and we conclude that $\bar z_k Y^k+z_k \bar Y^k=\alpha r$, so that (\ref{shrexp}) reads as

\begin{equation}\label{shrexpREW}
-(\log\det g)(r)-\lambda f(r) + \int_{U(n)}\left(F (Az) +\bar F(Az) \right) dA = \alpha r f'(r) 
\end{equation}

If $n=1$, then by Lemma \ref{lemradial}   the real part of $ \int_{U(n)}F (Az) dA$ is equal to $h+k \log r$, with $h, k$ constants. In this case, \eqref{shrexpREW} gives
$$-\log \left[f'(r)+rf''(r)\right]-\lambda f(r) + h + k \log r = \alpha r f'(r).$$

By derivating both sides of this equation with respect to $r$ and by multiplying by $r$
$$-\frac{f''(r)+(f''(r) r)'}{f'(r)+f''(r) r}r-\lambda r f'(r)  + k = \alpha (f'(r) r)' r,$$
by  $y(r) = rf'(r)$ and $\psi (r) = r(rf'(r))'$ 
we get
$$\dot\psi(y) = -\alpha \psi(y) + k+1 -\lambda y.$$
Then  \eqref{finalen=1} follows  by setting   $\mu=-\alpha$.
Moreover,  if  $\mu = 0$ equation \eqref{finalen=1} integrates and gives 
$$\psi (y)=-\lambda \frac{y^2}{2}+(k+1)y+c, \ c\in\R,$$
which by \cite[Lemma  2.1]{LSZext}  implies $g$ has constant scalar curvature and hence is  KE.

If $\mu\neq 0$,  one easily  integrates  \eqref{finalen=1}  
and gets \eqref{psiespln=12}. By  \eqref{psiespln=12} and by taking into account \cite[Lemma  2.1]{LSZext} 
we deduce that  $g$ is cscK iff it is flat iff $\nu=0$.

Let us now assume $n \geq 2$.
By applying again Lemma \ref{lemradial} to $\Phi = \int_{U(n)}F (Az) dA$, we get that  the real part of $ \int_{U(n)}F (Az) dA$ is constant and hence \eqref{shrexp} reads as
$$-\log\det g-\lambda f(r) + k =\alpha r f'(r).$$
If we derivate this equation  (with respect to $r$) and multiply both sides by $r$ we get
$$
r[-\log \det(g)]' - \lambda rf'(r) = \alpha r(r f'(r))'
$$
i.e.
$$
r[-\log \det(g)]' - \lambda y = \alpha \psi(y).
$$
Now, by using  \eqref{detmetric} one obtains
$$
-r[\log(y^{n-1}(r) \psi (y))]' + n - \lambda y = \alpha \psi (y).
$$

Since $\frac{d}{dr} = \frac{dy}{dr} \frac{d}{dy} = (rf'(r))' \frac{d}{dy} = \frac{\psi(y)}{r} \frac{d}{dy} $, we can rewrite the previous expression  as

$$\dot\psi(y) = -\left( \alpha + \frac{n-1}{y} \right) \psi (y) + n - \lambda y$$

which gives  \eqref{finale9}  by setting   $\mu=-\alpha$,

By integrating \eqref{finale9} one gets

\begin{equation}\label{psifond}
\psi(y) = \frac{e^{\mu y}}{y^{n-1}} \left[ \nu + \int (n - \lambda y)e^{-\mu y}y^{n-1} dy \right],
\end{equation}
for some constant $\nu\in\R$.

By taking  $\mu = 0$ in  \eqref{psifond} we get
$$\psi (y)=y+\frac{\nu+c}{y^{n-1}}-\frac{\lambda y^2}{n+1},$$
which, together with  \cite[Lemma  2.1]{LSZext}, implies that the metric $g$ is KE (with Einstein constant $2\lambda$).

If  $\mu \neq 0$ then \eqref{psiesplicita} follows, after a long but straight computation, by \eqref{psifond}
and by
$$
\int e^{-\mu y}y^{k} dy = \frac{e^{-\mu y}}{\alpha} y^k - \sum_{j=0}^{k-1} \frac{k!}{(k-j-1)! \mu^{j+2}} e^{-\mu y}y^{k-j-1}.
$$
Finally, by combining (16) with  \cite[Lemma  2.1]{LSZext}  one gets that $g$ is never KE unless it is flat and this happens exactly when  $\nu=0$ and $\mu=\lambda$.

In order to prove the last assertion of the Proposition, first notice that (\ref{psiesplicita}) rewrites as  
\begin{equation}\label{psiesplicita2}
\psi(y) = \frac{\nu e^{\mu y} + \frac{\lambda}{\mu} y^n +  \frac{\lambda - \mu}{\mu^{1+n}} \sum_{j=0}^{n-1} \frac{n!}{j!} \mu^j y^{j}}{y^{n-1}}.
\end{equation} 
If the metric  $g$ is defined at the origin, then both $y(r)=rf'(r)$ and $\psi (r) = r(rf'(r))'$ are  defined and vanish at  $r_{\inf}=0$ and 
\eqref{psiesplicita2} yields  $\nu = n! \frac{\mu - \lambda}{\mu^{1+n}}$.

Conversely, if $\nu = n! \frac{\mu - \lambda}{\mu^{1+n}}$ then (\ref{psiesplicita2}) implies that $\psi(y)=y + O(y^2)$ and  (by the Hartman-Grobman linearisation theorem, see also \cite{FIK}, Section 4.2) the differential equation $\frac{dy}{dt} = \psi(y)$ can be conjugated in a neighbourhood of $y=0$ to the linear equation $\frac{dz}{dt} = z(t)$ by a diffeomorphism $y = \Phi(z)$ satisfying $\Phi(0)=0$. Thus $\Phi(z) = z \tilde \Phi(z)$ for some smooth $\tilde \Phi$  and $y(t) = c e^t \tilde \Phi(c e^t)$. By $e^t = r$ and $y(r) = rf'(r)$ we finally get $f'(r) = c  \tilde \Phi(c r)$ which implies that $f(r)$ is smooth in $r=0$.
\end{proof}

\begin{rmk}\rm
A different proof of equation \eqref{psiesplicita} is obtained by  Feldman-Ilmanen-Knopf \cite[Section 3.2]{FIK}  
when the vector field $X$ is assumed to be gradient
(see also \cite{CAO} for the case of radial steady KRS, i.e. $\lambda=0$).
In fact, it is not hard to see that the vector field $Y$ appearing in the proof of Proposition \ref{mainprop} is a gradient vector field. However, in this paper we include the proof of Proposition \ref{mainprop} for reader's convenience and to make the paper as self contained as possible.
\end{rmk}

\section{The proof of Theorem \ref{mainteor}}\label{proofmainteor}

A finite or infinite dimensional complex space form $(S^{N}, g^{N}_c)$ is a manifold  of constant holomorphic sectional curvature $c$ and complex dimension $N\leq\infty$.
By the word \lq\lq induced'' we mean that the \K\ manifold  $(M, g)$  can be  \K\  immersed  into $(S^{N}, g^{N}_c)$, i.e. there exists a holomorphic map $\varphi :M\rightarrow S^{N}$ such that 
$\varphi^*g^{N}_c=g$ (see \cite{calabi} or the book \cite{LoiZedda-book} for an updated material on the subject).

If one assumes that   $(S^{N}, g^{N}_c)$ is  complete and simply-connected one has the corresponding three cases, depending on the sign of $c$:

- for $c=0$, $S^N=\C^N$ ($S^\infty=\ell^2(\C)$) and  $g^N_0$ is the flat metric with   associated \K\ form
$$
\omega_{0}=\frac{i}{2}\partial\bar\partial |z|^2,\  |z|^2=\sum_{j=1}^N|z_j|^2,\ N\leq\infty;
$$

- for $c<0$, $S^N=\C H^N$ is  the $N$-dimensional complex hyperbolic  space, namely the unit ball of $\C^N$ with the metric $g_c^N$ with associated  \K\ form 
$$
\omega_c=\frac{i}{2c}\partial\bar\partial\log (1-|z|^2);
$$

- for $c>0$, $S^N=\C P^N$ is  the $N$-dimensional complex projective space 
 and $g^N_c$  is the  metric with associated  \K\ form  $\omega_c$, 
 given in homogeneous coordinates by:
$$
\omega_c=\frac{i}{2c}\partial\bar\partial\log (|Z_0|^2 +\cdots +|Z_N|^2).
$$
 Notice that when $c=1$ (resp. $c=-1$) the metric 
$g_c^N$ is the standard Fubini-Study metric $g_{FS}$ (respectively hyperbolic metric $g_{hyp}$) of 
holomorphic sectional curvature $4$ (resp. $-4$).
Throughout the paper we will say  that a metric $g$ on a complex (connected) manifold is {\em projectively induced} if $(M, g)$ admits a \K\ immersion
into $(\CP^{\infty}, g_{FS})$. 

Let $\epsilon\in\{-1,0,1\}$
and define recursively the following function in $y$
\begin{equation}\label{Qk}
Q_1^\epsilon  (y):= y;\qquad Q_{k+1}^\epsilon (y) = (\epsilon y - k)Q_k^\epsilon (y) + {\dot Q}_k^\epsilon (y)\psi (y),
\end{equation}
with $\psi (y)$ given by \eqref{psi}.
\begin{lem}\label{calcrit}
Let $M$ be an $n$-dimensional ($n\geq 1$) complex manifold equipped with a \K\ metric $g$ with radial  \K\ potential  $f(r)$ which is real analytic in $(r_{\inf}, r_{\sup})$.
If  $(M, g)$ can be \K\ immersed into $(S^N, g^N_\epsilon)$ then 
the $Q_k^\epsilon (y)$ are nonnegative for $y\in (y_{\inf}, y_{\sup})$. 
Moreover,  if $r_{\inf}=0$ and $f(r)$ is defined in $[0, r_{\sup})$  the converse holds true.
\end{lem}
\begin{proof}
See \cite[Lemma 3.2]{LSZ} for a proof.
\end{proof}

\begin{proof}[Proof of Theorem \ref{mainteor}]
By the last part of Proposition \ref{mainprop}
the assumption \eqref{condpar} implies that the potential $f(r)$
is defined at the origin and thus
the metric $g$ is real-analytic on  $[0, r_{\sup})$ (see, e.g.  \cite[Corollary 1.3]{Kot} for a proof). 
We first show that $y(r)$ and all its derivatives w.r.t. $r$ are non-negative 
on  $[0, r_{\sup})$.

We only treat  the case $n\geq 2$ (the case $n=1$ is obtained similarly and it is omitted).
Under the assumption  \eqref{condpar} equation (\ref{psiesplicita}) reads as
\begin{equation}\label{psiesplicita3}
\psi(y) =  y + \sum_{k=1}^{\infty} \frac{\nu \mu^{n+k}}{(n+k)!} y^{k+1}
\end{equation}

Moreover, by $\nu, \mu > 0$ we have $\psi(y) \geq 0$.

We claim  that for every $k$ one has 

\begin{equation}\label{minca}
y^{(k)}(r)  = \frac{F_k(y(r))}{r^{k}}
\end{equation}
with 
\begin{equation}\label{derivateinduzione}
 F_k(y)=O(y^{k}), \ \ F_k(y) \geq 0.
\end{equation}

Formulae \eqref{minca} and \eqref{derivateinduzione} hold true for $k=1$ with $F_1(y) = \psi(y)\geq 0$ since $y'(r) = \frac{dy}{dt} \frac{dt}{dr} =  \frac{y'(t)}{r} = \frac{\psi(y)}{r}$, and $\psi(y) = O(y) \geq 0$ by \eqref{psiesplicita3}. 
Assuming now that  \eqref{minca} and (\ref{derivateinduzione}) are  true for some $k$, we have
$$y^{(k+1)}(r)  = \frac{d}{dr} \left( \frac{F_k(y)}{r^{k}} \right) =  \frac{\frac{dF_k}{dy}y'(r) r^{k} - F_k(y) k r^{k-1}}{r^{2k}} =$$
(recalling that $y'(r) = \frac{\psi(y)}{r}$)
\begin{equation}\label{equationFkinduc}
= \frac{\frac{dF_k}{dy}\psi(y) - k F_k(y)}{r^{k+1}}.
\end{equation}
This shows that (\ref{minca}) holds true also for $k+1$, with
$$F_{k+1}(y) = \frac{dF_k}{dy}\psi(y) - k F_k(y),$$ 
By inserting (\ref{psiesplicita3}) and $F_k(y) = \sum_{j=k}^{\infty} a^k_j y^j$ in this recursion relation it is now easy to see that  $F_{k+1}(y)=O(y^{k+1})$ and $F_{k+1}(y) \geq 0$, which concludes the proof by induction that  \eqref{minca} and \eqref{derivateinduzione} are true for every $k \geq 1$ and then that the derivatives $y^{(k)}(r)$ are non-negative on  $[0, r_{\sup})$ for $k \geq 1$.

By $r f'(r) = y(r) = \sum_{k=1}^{\infty} \frac{y^{(k)}(0)}{k!} r^k$, one immediately deduces that the derivatives $f^{(k)}(r)$ are non-negative for $k \geq 1$. 

Now we claim that the functions $Q^{\epsilon}_k(y)$ defined in (\ref{Qk}) satisfy $Q_{k}^0 (y) = r^k f^{(k)}(r)$ for every $k \geq 1$: this will prove that these functions are non-negative and
by Lemma \ref{calcrit}, $(M, g)$ can be \K\ immersed into $\ell^2(\C)$ and hence also in $\C P^{\infty}$ by a result of Calabi  \cite{calabi} (this proves (a) of Theorem \ref{mainteor}).

In order to prove the claim, notice that (\ref{Qk}) reads

\begin{equation}\label{QkZERO}
Q_1^0  (y):= y;\qquad Q_{k+1}^0 (y) = - k Q_k^0 (y) + {\dot Q}_k^0 (y)\psi (y)
\end{equation}

Since $y = rf'(r)$, we have $Q_{1}^0 (y) = r f^{'}(r)$; assuming now that $Q_{k}^0 (y) = r^k f^{(k)}(r)$ is true for some $k$, we have by (\ref{QkZERO})

$$Q_{k+1}^0 (y) = - k r^k f^{(k)}(r)+ \frac{d}{dr} \left( r^k f^{(k)}(r) \right) \frac{dr}{dy} \psi (y) =$$
(by using $\frac{dr}{dy} = \frac{r}{\psi}$)
$$= - k r^k f^{(k)}(r)+ k r^{k} f^{(k)}(r) + r^{k+1} f^{(k+1)}(r) = r^{k+1} f^{(k+1)}(r) $$

which proves the claim and ends the proof of part (a) of Theorem \ref{mainteor}.

\bigskip

Assume now that the parameters of the   radial KRS $(g, X)$  
satisfy
\begin{equation}\label{condpar2}
\nu = \frac{n!(n+1-\lambda)}{(n+1)^{n+1}},\ \lambda\leq 0,\ \mu=n+1, \ k=0.
\end{equation}
To prove (b) of Theorem \ref{mainteor} we need  to show that the \K\ manifold $(M, g)$ can be \K\ immersed into  
$(\CH^\infty, g_{hyp})$. Indeed this would imply  that it can be \K\ immersed into the infinite dimensional flat space by a result of Bochner \cite{boc} and into any infinite dimensional  complex projective space by  \cite[Lemma 8]{DHL}.

Notice that by (a) the radial potential $f(r)$
of the  metric $g$   of the  family of KRS given by \eqref{condpar2} is real-analytic on $[0, r_{\sup})$. 
Hence, by Lemma \ref{calcrit}, we must prove that $Q_k^{-1}(y)$ is non-negative $\forall k\in\Z^+$ on $[0, y_{\sup}).$
The proof is by induction on $k$.
We treat only the case $n\geq 2$ (the case $n=1$ is treated similarly).
First, let us notice that \eqref{psiesplicita3} under the assumptions \eqref{condpar2}  writes
$$\psi(y)=y+\frac{n!(n+1-\lambda)}{n+1}\sum_{k=1}^\infty\frac{(n+1)^k}{(n+k)!}y^{k+1}=y+\left(1-\frac{\lambda}{n+1}\right)y^2+O(y^3).$$
We now assume by induction that the coefficients in the expansion of $Q_k^{-1}(y)$ are all nonnegative and that it vanishes at $y=0$ with order greater or equal to $k$. This property is clearly verified  for $k=1$, since $Q_1^{-1}(y)=y$ by construction. Then, if $Q_k^{-1}(y)=\sum_{j\geq k} a_jy^j$, by \eqref{Qk}  with $\epsilon =-1$,  we get
\begin{multline*}
Q_{k+1}^{-1}(y)=-(y+k)Q_k^{-1}(y)+{\dot Q}_k^{-1}(y)\psi(y)=\\
-\sum_{j\geq k} a_jy^{j+1}-\sum_{j\geq k} ka_jy^j+\sum_{j\geq k}j a_jy^{j-1}\left(y+\left(1-\frac{\lambda}{n+1}\right)y^2+O(y^3)\right)=\\
\sum_{j\geq k+1}(j-k) a_jy^{j}+\sum_{j\geq k}\left(j-\frac{\lambda j}{n+1}-1\right)a_jy^{j+1}+O(y^{k+2}).
\end{multline*}
and we notice  that all the coefficients of $O(y^3)$ in the second line are positive, so the coefficients in $O(y^{k+2})$ in the last line are all nonnegative. 
\end{proof}
\begin{rmk}\rm
Combining Lemma \ref{calcrit} with the fact that $(M, g)$ cannot be \K\ immersed into $\C H^N$, with $N<\infty$ 
(see \cite{LM} for a proof),   we deduce that 
the number of positive  $Q_k^{-1}(y)$ in the proof of the previous  theorem is forced to be  infinite.
\end{rmk}

%\begin{rmk}\rm
%It is also  worth noticing that one can find parameters in the family \eqref{condpar} such that 
%the corresponding  KRS  cannot be immersed into 
 %$(\CH^\infty, g_{hyp})$. For example, take $\lambda=-n/2$, $\mu=1/2$ e $\nu=(n+1)! 2^{n}$. In this case, we easily calculate 
%$Q_2^{-1}(y)=-y^2/2+O(y^3)$ and the assertion follows again by Lemma \ref{calcrit}.
%\end{rmk}

\section{Complete KRS solitons and the proof of Theorem \ref{mainteor3}}\label{compKRS}

Let us now investigate the completeness of our radial metrics.

The matrix of a radial metric with radial potential $f(r)$ is given by (see \eqref{metric})

$$
g_{i\bar j}=f'(r)\delta_{ij}+f''(r) \bar z_i z_j
$$
Take now a curve $\gamma(s) = (z(s), 0, \dots, 0)$, $z(s) \in \R$, $z'(s) > 0$, where $s \in (s_1, s_2)$. By definition, its length is given by

$$l(\gamma) = \int_{s_1}^{s_2} \sqrt{g_{1 \bar 1} z'^2(s)} ds = \int_{s_1}^{s_2} \sqrt{f'(z^2(s)) + f''(z^2(s)) z^2(s)} z'(s) ds$$ 

i.e. by the change of variable $z= z(s)$

$$l(\gamma) = \int_{z_1}^{z_2} \sqrt{f'(z^2) + f''(z^2) z^2} dz$$ 

where we have set $z_1 = z(s_1)$, $z_2 = z(s_2)$.

Now, set $r = z^2$ so that $dz = \frac{1}{2 \sqrt r} dr$ and

$$l(\gamma) = \frac{1}{2}  \int_{r_1}^{r_2}  \sqrt{\frac{f'(r) + f''(r) r}{r}} dr = \frac{1}{2}  \int_{r_1}^{r_2}  \sqrt{\frac{(rf')'}{r}} dr $$ 

Now, in order to rewrite this integral in terms of the functions $y(r) = rf'(r)$ and $\psi(y(r)) = r(rf'(r))'$, we make the change of variable $y = y(r)$. Notice that $\frac{dy}{dr} = (rf'(r))' = \frac{\psi f((r))}{r}$, so we get
$$l(\gamma) = \frac{1}{2}  \int_{y_1}^{y_2}  \sqrt{\frac{\psi}{r^2}} \frac{r}{\psi} dy = \frac{1}{2}  \int_{y_1}^{y_2}  \sqrt{\frac{1}{\psi}} dy$$
where $y_1=y(r_1)$ and $y_2=y(r_2)$.

Therefore we deduce that a radial metric corresponding to the function $\psi(y)$ defined on $[0, y_{sup})$ is complete if and only if
\begin{equation}\label{propcompl}
\lim_{y_2 \rightarrow y_{\sup} } \int_{y_1}^{y_2}  \sqrt{\frac{1}{\psi}} dy = + \infty .
\end{equation}

\begin{example}\rm
Let $(g, X)$ be the KRS of the  complex manifolds $M$ of  dimension $n\geq 2$ given by Theorem \ref{mainteor}. 
As we have seen at the beginning of the proof of Theorem \ref{mainteor}, assumption  \eqref{condpar} implies that $\psi$ is given by (\ref{psiesplicita3}).
Then, for every $t_0 \in \R$ and $y_0 > 0$, the function

\begin{equation}\label{Psigrande}
\Psi(y) := \int_{y_0}^y \frac{dy}{\psi(y)}
\end{equation}
(which gives an implicit solution $\Psi(y) = t-t_0$ of the Cauchy problem $\frac{dy}{dt} = \psi(y)$, $y(t_0) = y_0$) is defined for every $y > 0$ since, by (\ref{psiesplicita3}), under the assumption $\nu > 0, \mu > 0$ the denominator $\psi(y)$ of the integrand in (\ref{Psigrande}) vanishes only for $y=0$ and then the integral (\ref{Psigrande}) is finite for every $y > 0$.

Moreover, by (\ref{psiesplicita})  we have
$$\lim_{y \rightarrow + \infty}\int_{y_0}^{+ \infty} \sqrt{ \frac{y^{n-1} dy}{\nu \left[ e^{\mu y} - \sum_{j=0}^{n} \frac{\mu^j}{j!} y^{j} \right] + y^n} }  < +\infty$$
because the integrand goes to zero as $\sqrt{\frac{y^{n-1}}{e^{\mu y}}}$, and the integral $\int_{y_0}^{+ \infty} \sqrt{\frac{y^{n-1}}{e^{\mu y}}} dy$ converges.
Thus,  by \eqref{propcompl},  the metric is not complete.
\end{example}

\begin{example}\rm\label{filippofamily}
Let $(g, X)$ be the KRS of the  complex manifold $M$ of  dimension $n\geq 2$
associated  to $(\nu, \mu)\in\R^2$ satisfying
\begin{equation}\label{condvarie}
\nu = n! \frac{\mu - \lambda}{\mu^{1+n}}, \ \ \lambda = \mu - n -1 < 0, \  \mu < 0.
\end{equation}
We want to show that $M=\C^n$ and that the metric $g$ is complete.

First, by assumption \eqref{condpar} the metric is defined at the origin and $\psi$ is given by (\ref{psiesplicita3}).

Moreover, the function $\psi$ vanishes only for $y=0$: indeed, assume by contradiction that under the assumptions $\nu = n! \frac{\mu - \lambda}{\mu^{1+n}}, \  \ \mu <0, \ \ \lambda < 0$ there exists another positive zero $y=a$ for $\psi$. Then, by (\ref{finale9}) we get $\psi'(a) = n - \lambda a > 0$, which is not possible (if $\psi$ starts positive from zero, then it must be decreasing in a neighbourhood of the positive zero $a$).
It follows that $\psi$ is defined and positive for $y \in (0, +\infty)$.
Now, the implicit definition of the solution $y(t)$, i.e. $\int_{y_0}^{y} \frac{dy}{\psi(y)} = t - t_0$, rewrites
$$ \int_{y_0}^{y} \frac{y^{n-1}}{\nu \left[ e^{\mu y} - \sum_{j=0}^{n} \frac{\mu^j}{j!} y^{j} \right] + y^n} dy = t - t_0$$
and then, since $\mu < 0$, one immediately sees that the integrand goes to zero as $1/y$ for $y \rightarrow + \infty$, so it diverges and $t_{\sup} = + \infty$. 
In terms of $r= e^t$ we get $(r_{\inf}, r_{\sup}) = (0, +\infty)$ and the metric is defined on all $\C^n$.

As for completeness, by \eqref{propcompl} we need to check that the integral
\begin{equation}\label{intcompl}
\int_{y_0}^{+\infty} \sqrt{ \frac{y^{n-1}}{\nu \left[ e^{\mu y} - \sum_{j=0}^{n} \frac{\mu^j}{j!} y^{j} \right] + y^n} } dy
\end{equation}
diverges.
But this is clear since, as already observed above, for $y \rightarrow + \infty$ the function $\frac{1}{\psi}$ goes to zero as $1/y$, so that the integrand in (\ref{intcompl}) goes to zero as $1/\sqrt y$, and then the conclusion follows by $\int_{y_0}^{+\infty}\frac{1}{\sqrt y} dy =  [2\sqrt y]_{y_0}^{+ \infty} = + \infty$.
\end{example}

\begin{rmk} \rm
Let us notice that this soliton is an example of the complete expanding soliton on all of $\C^n$ found by Cao in \cite{CAO} characterized by the values of the parameters (in our notation) $\lambda = -1$, $\mu < 0, \nu = n! \frac{\mu- \lambda}{\mu^{1+n}}$.
\end{rmk}

\begin{proof}[Proof of Theorem \ref{mainteor3}]
Take the complete non trivial  KRS $(\C^n, g_{\mu})$ in Example \ref{filippofamily} and set
\begin{equation}\label{psinew}
\psi(y,\mu):=\psi (y)=y+ \sum_{j=2}^\infty\frac{(n+1)!}{(n+j-1)!}\mu^{j-2}y^j.
\end{equation}
In order to prove the theorem we will show that for a suitable  choice of $\mu$  the \K\ manifold   $(\C^n, g_\mu)$ can be \K\ immersed into $\ell^2(\C)$ (and hence into $\C P^{\infty}$).

By using Weierstrass M-test one sees that 
$\frac{\de^h \psi}{\de y^h}$  are continuous with respect to  $\mu$, for all $h\in\N$ in the interval $[-1, 1]$.
Then, by definition of $Q^0_k$, namely
\begin{equation}\label{Q0new}
Q^0_{k+1}(y,\mu)=\dot {Q^0_k} \psi(y,\mu)-kQ^0_k(y,\mu),
\end{equation} 
every derivative  w.r.t.  $y$ of $Q^0_k(y,\mu)$ is  continuous  w.r.t. $\mu$.

Since $\psi(y,0)=y+y^2$
and 
\begin{equation}\label{Qkcontr}
Q^0_k(y,0)=(k-1)! y^k,
\end{equation}
we deduce that  for every $k\in\N$, there exists a real negative constant $-1\leq\epsilon_k <0$ such that
$$\frac{\de^k Q^0_k}{\de y^k}\Big|_{(0,\mu)}>0, \quad\forall\mu\in (\epsilon_k, 0).$$
Moreover, we claim that    
$$I:=\bigcap_{k\in\N} (\epsilon_k, 0)\neq\emptyset.$$
Indeed, if by contradiction  $\lim_{k\to \infty} \epsilon_k =  0$ then one can easily  find a sequence $\mu_k\in [-1, 0)$ such that 
$\lim_{k\to \infty} \frac{\de^k Q^0_k}{\de y^k}\Big|_{(0,\mu_k )}=0$.
On the other hand,  by \eqref{Qkcontr} we deduce that    $\lim_{k\to \infty} \frac{\de^k Q^0_k}{\de y^k}\Big|_{(0,0 )}= +\infty$, yielding the desired contradiction and proving the claim.
 
 Finally notice that by \eqref{psinew} and \eqref{Q0new} we have  that $Q^0_k(y)=O(y^k)$ independently  of $\mu$ and then 
 by Lemma \ref{calcrit} we deduce that  $g_{\mu}$  with  $\mu\in I$ is induced by $\ell^2(\C)$.

\end{proof}

\section{Projectively induced KRS and the proof of Theorem \ref{mainteor2}}\label{projindKRS}
We start by describing  some necessary conditions for a radial KRS to be projectively induced.

\begin{prop}\label{proph}
Let $M$ be an $n$-dimensional complex manifold with $n\geq 2$
and let  $g$ be the \K\ metric of the radial KRS  corresponding to the function $\psi(y)$, $y\in (y_{\inf}, y_{\sup})$ with solitonic constant $\lambda$.
Let us assume that $g$ is projectively induced.  Set $y_{\inf}:=h$. Then the following facts hold true:
\begin{itemize}
\item [(i)]
$\psi(h) = 0$;
\item [(ii)]
$h\in\N$;
\item [(iii)]
$\dot\psi(h)\in\Z$;
\item [(iv)]
if  $h\neq 0$ then $\lambda\in\Q$.
\end{itemize}
\end{prop}
\begin{proof}
Throughout  the proof we will denote the $Q_k^{+1}(y)$ appearing in  (\ref{Qk}) simply by $Q_k(y)$.
Assume by contradiction that $\psi(h) \neq 0$. Then by Lemma \ref{lemmalimpsi} one has
$h= 0$.
Thus,  by \eqref{finale9} when $n\geq 2$ and the fact that $\psi (y)>0$ we see that 
\begin{equation}\label{psiprimo}
\lim_{y \rightarrow h^+} \dot\psi(y)=\lim_{y \rightarrow 0^+} \dot\psi(y)= - \infty
\end{equation}

By deriving (\ref{Qk}) with  $\epsilon =1$ and  $k=1$, we get

$${\dot Q}_2(y) = 2y-1 +\dot\psi(y)$$

which combined with (\ref{psiprimo}) yields $\lim_{y \rightarrow 0^+}{\dot Q}_2(y)= - \infty$.

By (\ref{Qk}) with  $\epsilon =1$ and  $k=2$
$$Q_{3}(y) = (y-2)Q_2^1(y) + \psi(y) 
{\dot Q}_2(y).$$

We then immediately deduce that $\lim_{y \rightarrow 0^+} Q_3(y) = - \infty$, in contrast with 
Lemma \ref{calcrit}, since the metric is projectively induced, and (i) is proved.

Notice now that
by  (i) 
and  (\ref{Qk}) (with $\epsilon=1$) 
one can prove  by induction on $k\in\Z^+$ that 
\begin{equation}\label{Qkindu}
Q_{k+1}(h) = (h- k)(h - k +1) \cdots ( h-1)h.
\end{equation}

Assume  by contradiction that  $h\notin \Z$. Then by \eqref{Qkindu}  there exists $k$ such that 
$Q_{k+1}(h) <0$ again in contrast with the projectively induced assumption on $g$. Thus (ii) is proved.
 
 By \eqref{Qkindu}  we also deduce that 
 \begin{equation}\label{Qsol}
Q_j(h) = 0, \forall j \geq h+1.
\end{equation}
We claim that
\begin{equation}\label{Qkindu2}
{\dot Q}_{h+j}(h) = (\dot\psi(h) - 1)(\dot\psi(h)  - 2) \cdots (\dot\psi(h)  - j+1) {\dot Q}_{h+1}(h),
\end{equation}
for all  $j\geq 2$.

Indeed, by deriving (\ref{Qk}) we have
\begin{equation}\label{Qkder}
{\dot Q}_{h+2}(y) = Q_{h+1}(y) + (y-h-1){\dot Q}_{h+1}(y) + \psi (y){\ddot Q}_{h+1}(y) + \dot\psi(y){\dot Q}_{h+1}(y)
\end{equation}
and the assertion follows for $j=2$ by letting $y =h$ and using $Q_{h+1}(h)=0$ (by \eqref{Qsol}) and $\psi (h)=0$ (by (i)). 
Assuming that (\ref{Qkindu2}) is true for some $j$,  by deriving (\ref{Qk}) w.r.t. $y$   we get
\begin{equation}\label{Qkder}
{\dot Q}_{h+j+1}(y) = Q_{h+j}(y) + (y-h-j){\dot Q}_{h+j}(y) + \psi (y){\ddot Q}_{h+j}(y)+ \dot\psi(y){\dot Q}_{h+j}(y)
\end{equation}
By taking  $y=h$ we see that  
$${\dot Q}_{h+j+1}(h) = ( \dot\psi(h)-j){\dot Q}_{h+j}(h),$$
which, together with the inductive assumption, proves our claim.

By (\ref{Qk}) and its derivative (with $k=j$) with respect to $y$ and taking into account (i), i.e.  
$\psi(h)=0$, on easily gets
$$Q_j(h) = Q_{j-1}(h)(h-j+1)$$
$${\dot Q}_j(h) = {\dot Q}_{j-1}(h)(h-j+1) + Q_{j-1}(h) + \dot\psi(h){\dot Q}_{j-1}(h).$$
Combining these two equalities and using $Q_1(y) = y$  we find
$$(Q_j + \dot\psi {\dot Q}_j)(h) = (h + \dot\psi (h))(h-1+\dot\psi (h)) \cdots (h-j+1+\dot\psi (h)).$$
Taking $j=h+1$ and using $Q_{h+1}(h)=0$  we get
$${\dot Q}_{h+1}(h) = (h + \dot\psi (h))(h-1+\dot\psi (h)) \cdots (1+\dot\psi (h)).$$
Now if  by contradiction (iii) is false, i.e. $\dot\psi(h)\notin\Z$, we  get ${\dot Q}_{h+1}(h)\neq 0$.
Thus by \eqref{Qkindu2}
we deduce $\dot Q_{h+j}(h)<0$ for some $j$, which combined with 
\eqref{Qsol} implies  that $Q_{h+j}(y)<0$ on a right neighborhood of $h$, in contrast with fact that 
$g$ is projectively induced.

By combining (i) and  (\ref{finale9}) (here we are using the assumption that $n\geq 2$)
 one deduces that $\dot\psi (h)= n - \lambda h$ and hence (iv) readily follows by (ii) and  (iii).
\end{proof}

Before proving Theorem \ref{mainteor2} we recall the definition of $c$-stable projectively induced metric.

\begin{defin}\label{defst}
Let $c>0$.
 A \K\  metric $g$  is said to be  {\em $c$-stable projectively induced} if there exists $\epsilon >0$ such that $\alpha g$ is induced  by $(\C P^\infty, g_c^\infty)$
for all $\alpha \in (1-\epsilon, 1+\epsilon)$.   A \K\  metric $g$  is said to be unstable if it is not $c$-stable projectively 
induced for any $c>0$. When $c=1$ we simply say that $g$ is {\em stable-projectively induced}.
\end{defin}

\begin{rmk}\rm\label{rmkstable}
Notice  that a  \K\ metric which can be \K\ immersed into a non-elliptic (finite or infinite dimensional) complex space form is authomatically $c$-stable projectivelly induced  (the reader is referred to \cite{LSZ} for details). 
\end{rmk}

\begin{proof}[Proof of Theorem \ref{mainteor2}]
Without loss of generality we can assume that $g$ is induced by $(\CP^{\infty}, g_{FS})$
and hence $g$ is stable projectively induced.  Therefore,  by multiplying the metric 
by a suitable  positive constant $\beta$, we can assume that   $\beta g$ is still projectively induced and  with  solitonic constant 
$\frac{\lambda}{\beta}\in \R\setminus \Q$.
Hence  by (i) and (iv) of  Proposition \ref{proph} $y_{\inf}=0$ and $\psi (y_{\inf})=0$.
Thus, as  seen in the last part of the proof of Proposition \ref{mainprop}, $f(r)$ is defined at $r_{\inf}=0$. 

\end{proof}

\appendix
\section{The proofs of Lemma \ref{lemradial} and Lemma \ref{lemradial2} }\label{proofslemma}

\begin{proof}[Proof of Lemma \ref{lemradial}]
If we derivate the equality $G(z)=\Phi(z) + \bar \Phi(z)$ with respect to $z_j$ and $\bar z_k$ we obtain

\begin{equation}\label{Gjk}
0= \frac{\partial^2 G}{\partial x_j \partial x_k} z_k \bar z_j + \frac{\partial G}{\partial x_j} \delta_{jk}
\end{equation}

Now,  if $n \geq 2$ we can take $k \neq j$ in this equation and deduce $\frac{\partial^2 G}{\partial x_j \partial x_k} =0$, so that $\frac{\partial G}{\partial x_j } = \Gamma_{j}(x_j)$, for some smooth function $\Gamma_j(x_j)$, $j=1, \dots n$. This combined with (\ref{Gjk}) for $k=j$ yields

\begin{equation}\label{Gjk2}
0= \Gamma_{j}'(x_j) x_j + \Gamma_{j}(x_j), \forall j=1, \dots n,
\end{equation}

i.e.,

\begin{equation}\label{Gjk3}
(\Gamma_{j}(x_j) x_j)' =0, \ \forall j=1, \dots n, 
\end{equation}

and then

$$\Gamma_{j}(x_j) = \frac{\partial G}{\partial x_j } = \frac{c_{j}}{x_j},  \ \forall j=1, \dots n,$$

from which  \eqref{forlemmaradial} follows immediately.

If $n = 1$, equation (\ref{Gjk}) writes $G''(x) x + G'(x) = 0$, which can be treated as equation (\ref{Gjk2}) to deduce that $G'(x) = \frac{c}{x}$ (for  some constant $c$) and obtain the same conclusion.

The last assertion in the statement follows immediately by noticing that \eqref{forlemmaradial}  with 
$G$ rotation invariant can be satisfied only if the $c_j$'s vanish or $n=1$.  
\end{proof}

\begin{proof}[Proof of Lemma \ref{lemradial2}]
By deriving \eqref{forlemma2} with respect to $z_l$ and $\bar z_l$ we get

\begin{equation}\label{Yz}
\frac{\partial Y_l}{\partial z_l} + \frac{\partial \bar Y_l}{\partial \bar z_l} = \phi''(r) |z_l|^2 + \phi'(r), \ l=1, \dots n.
\end{equation}

Since the right-hand side is a rotation invariant function and $\frac{\partial Y_l}{\partial z_l}$ is holomorphic, we can apply Lemma \ref{lemradial} to deduce that 
$$\frac{\partial Y_l}{\partial z_l} = \sum_{j=1}^n c_{lj} \log z_j + d_l, \ l=1, \dots ,n$$ 
for some $c_{lj}, d_l \in \C$.
Then \eqref{Yz}  writes

$$
\sum_{j=1}^n (c_{lj} \log z_j + \bar c_{lj} \log \bar z_j)  + 2 \tilde d_l  = \phi''(r) |z_l|^2 + \phi'(r), \ l=1, \dots ,n,
$$

where $2\tilde d_l = d_l + \bar d_l$.

Then  we deduce that $c_{lj} \in \R$, and 

$$
\sum_{j=1}^n c_{lj} \log |z_j|^2 + 2 \tilde d_l  = \phi''(r) |z_l|^2 + \phi'(r), \ l=1, \dots ,n,
$$

By setting $x_j = |z_j|^2$, we can rewrite this equation as

$$
\frac{\partial}{\partial x_l} (\phi'(r) x_l) = \sum_{j=1}^n c_{lj} \log x_j + 2 \tilde d_l,  \ l=1, \dots ,n,
$$

and, by integrating, we get

\begin{equation}\label{equationLemma25}
\phi' (r)x_l = x_l \sum_{j=1}^n c_{lj} \log x_j + 2 \tilde d_l x_l + F_l(x),  \ l=1, \dots ,n
\end{equation}

where $F_l(x)$ does not depend on $x_l$.

If $n \geq 2$ by derivating (\ref{equationLemma25}) with respect to $x_k$, $k \neq l$, we get

$$
\left( \phi''(r) -  \frac{c_{lk}}{x_k} \right) x_l = \frac{\partial F_l}{\partial x_k}
$$

which, since $\frac{\partial F_l}{\partial x_k}$ does not depend on $x_l$, implies  that $F_l$ is a constant, say $f_l$, and $ \phi''(r) =  \frac{c_{lk}}{x_k}$. But, since $\phi'' = \phi''(x_1 + \cdots + x_n)$, this last equality implies that $c_{lk}=0$  for $k \neq l$. Then (\ref{equationLemma25}) becomes

\begin{equation}\label{equationLemma25rew}
\phi' (r)x_l = c_{ll} x_l \log x_l + 2 \tilde d_l x_l + f_l
\end{equation}

Since $\phi' = \phi'(x_1 + \cdots + x_n)$, this equality implies that $\phi'$ is a constant, and then that $\phi(r) = \alpha r$ as desired.

For $n=1$, we have the analogous of  (\ref{equationLemma25rew}) with $\phi'$ depending on one variable $x$ only, that is 

$$
x \phi'(x)= c x \log x + 2 \tilde d x + f,
$$

By derivating and setting $2\alpha=c+2\tilde d$

$$
\phi' (x)+ x \phi''(x)= c \log x + 2\alpha
$$

and then by (\ref{Yz})

$$
\frac{\partial Y}{\partial z} + \frac{\partial \bar Y}{\partial \bar z} =  c \log |z|^2 + 2\alpha
$$

which by the holomorphicity of $\frac{\partial Y}{\partial z}$ implies

$$
\frac{\partial Y}{\partial z} =  c \log z+ \alpha
$$

Then $Y = c \int \log z +  \alpha z + k$. Combined with the assumption $\bar z Y + z \bar Y = \phi(r)$ this yields
$$c \bar z \int \log z + \bar c z \int \log \bar z + \alpha  |z|^2 + (k \bar z + \bar k z) = \phi(r)$$
and this can hold true only if $c = k = 0$, so that $\phi(r) =\alpha r$, as desired. 
\end{proof}

\end{document}